\documentclass{amsart}

\usepackage{amssymb}
\usepackage{amsrefs}
\usepackage{overpic}
\usepackage{enumitem}   
\usepackage{graphicx} 
\usepackage{xcolor}
\usepackage[hidelinks]{hyperref}

\graphicspath{{Figures/}} 

\newtheorem{theorem}{Theorem}
\newtheorem{proposition}{Proposition} 
 
\newtheorem{remark}{Remark} 
\newtheorem{lemma}{Lemma} 
 
\newtheorem{example}{Example}

\begin{document}\title[New stability criteria for periodic planar Lotka-Volterra systems]{New stability criteria for periodic planar Lotka-Volterra systems}
	
\author[Paulo Santana]
{Paulo Santana$^1$}
	
\address{$^1$ IBILCE--UNESP, CEP 15054--000, S. J. Rio Preto, SP, Brazil}
\email{paulo.santana@unesp.br}
	
	
	
\begin{abstract}
	We present new a stability result for $T$-periodic solutions of the periodic predator–prey Lotka–Volterra model, based on boundaries for the average of the coexistence states. Our result complements previous one in the literature.
\end{abstract}
	
\maketitle
	
\section{Introduction}\label{Sec1}

Consider the planar non-autonomous periodic Lotka–Volterra system given by
\begin{equation}\label{1}
	\dot u=u(a(t)-b(t)u-c(t)v), \quad \dot v=v(d(t)+e(t)u-f(t)v),
\end{equation}
with all coefficients being continuous and $T$-periodic functions, $T>0$. Suppose also that $b$, $c$, $e$ and $f$ are strictly positive functions. Under these conditions system \eqref{1} models, in the positive quadrant $\mathbb{R}^2_+=\{(u,v)\in\mathbb{R}^2\colon u>0,\;v>0\}$, a predator-prey interaction between two species in a $T$-periodic environment. In the last decades, there is a great effort of R. Ortega and coauthors for the understanding of the $T$-periodic orbits of system \eqref{1}. In what follows, we present a brief survey on the obtained results. System \eqref{1} can have three types $T$-periodic non-negative solutions $(u(t),v(t))$.
\begin{enumerate}[label=(\alph*)]
	\item The \emph{trivial state}, given by $u(t)=v(t)=0$ for every $t\in[0,T]$;
	\item The \emph{semi-trivial state}, given by: 
	\begin{enumerate}[label=(\roman*)]
		\item $u(t)>0$, $v(t)=0$ for every $t\in[0,T]$, or
		\item $u(t)=0$, $v(t)>0$ for every $t\in[0,T]$;
	\end{enumerate}
	\item The \emph{coexistence state}, given by $u(t)>0$ and $v(t)>0$ for every $t\in[0,T]$.
\end{enumerate}
Observe that the trivial and semi-trivial states represents the extinction of at least one of the species.

The question about the \emph{existence} of coexistence states was completely solved by the authors in \cite{GomOrtTin1996}. More precisely, let 
	\[\lambda=\frac{1}{T}\int_{0}^{T}a, \quad \mu=\frac{1}{T}\int_{0}^{T}d.\]
It follows from \cite{GomOrtTin1996} that \eqref{1} has a semi-trivial state of type $(i)$ (resp. $(ii)$) if, and only if, $\lambda>0$ (resp, $\mu>0$). Moreover, in this case the semi-trivial state $(\theta_\lambda(t),0)$ (resp. $(0,\theta_\mu(t)$) is given by the unique positive $T$-periodic solution of
	\[\dot u=u(a(t)-b(t)u) \quad \bigl(\textnormal{resp. } \dot v=v(d(t)-f(t)v)\bigr).\] 
In particular, if it exists then it is unique (of its type). Moreover, it follows from \cite[Section $4$]{GomOrtTin1996} that the following statements hold (for a precise definition of \emph{linear} and \emph{asymptotic} stability, see Appendix~~\ref{AppA}).
\begin{enumerate}[label=(\alph*)]
	\item The trivial state $(0,0)$ is linearly stable if, and only if, $\lambda\leqslant0$ and $\mu\leqslant0$. Moreover, in this case it is also asymptotically stable in $\mathbb{R}^2_+$.
	\item The semi-trivial state $(\theta_\lambda,0)$ is linearly stable if, and only if, $\lambda>0$ and 
		\[\mu\leqslant-\frac{1}{T}\int_{0}^{T}e\theta_\lambda.\]
	Moreover, in this case it is also asymptotically stable in $\mathbb{R}^2_+$.
	\item The semi-trivial state $(0,\theta_\mu)$ is linearly stable if, and only if, $\mu>0$ and 
	\[\lambda\leqslant\frac{1}{T}\int_{0}^{T}c\theta_\mu.\]
	Moreover, in this case it is also asymptotically stable in $\mathbb{R}^2_+$.
\end{enumerate}
It can be proved \cite[Theorem $4.1$]{GomOrtTin1996} that for a coexistence state to exist it is necessary to the trivial and semi-trivial states (when it exist) to be linearly unstable. That is, it is necessary that
\begin{equation}\label{17}
	\mu>-\frac{1}{T}\int_{0}^{T}e\theta_\lambda, \quad \lambda>\frac{1}{T}\int_{0}^{T}c\theta_\mu.
\end{equation}
On the other hand, it also follows from \cite[Section $5$]{GomOrtTin1996} that \eqref{17} is as well a sufficient condition and thus we have a complete characterization on the existence of coexistence states for system \eqref{1}. Moreover, it follows from \cite[Theorem $5.1$]{GomOrtTin1996} that system \eqref{1} always has at most a finite number of coexistence states.

With the existence completely characterized, the question turned to the finiteness, uniqueness and stability of the coexistence states. For this end, before we enunciate the obtained results, we introduce some technical notations.

Let $\mathbb{T}_T=\mathbb{R}/T\mathbb{Z}$ be endowed with the push-forward measure of the Lebesgue measure on $\mathbb{R}$ (see Bogachev \cite[Section $3.6$]{Bogachev}). That is, the measure $\sigma$ such that $\sigma(a,b)=b-a$, for $0\leqslant a\leqslant b\leqslant T$. Given $p\in[1,\infty]$, let $L^p=L^p(\mathbb{T}_T)$ denote the usual Lebesgue $L^p$-space associated with $\mathbb{T}_T$. Since $\mathbb{T}_T$ has finite measure, we observe that
	\[L^\infty\subset L^{p_2}\subset L^{p_1}\subset L^1,\]
where $1\leqslant p_1\leqslant p_2\leqslant\infty$. In particular, since the coefficients of \eqref{1} are continuous, it follows that they lie in $L^\infty$ and thus also in $L^p$, for every $p\in[1,\infty]$. Let $\varphi\in L^p$. For simplicity we denote $||\varphi||_p=||\varphi||_{L^p}$, where $||\cdot||_{L^p}$ is the usual norm of the Banach space $L^p$. Let also
	\[\varphi_L=\min\limits_{t\in[0,T]}\{\varphi(t)\}, \quad \varphi_M=\max\limits_{t\in[0,T]}\{\varphi(t)\}, \quad \overline{\varphi}=\frac{1}{T}\int_{0}^{T}\varphi,\]
be the minimum, maximum and average of $\varphi$. 

About the uniqueness of the coexistence states of \eqref{1}, consider the statements
\begin{equation}\label{18}
	\overline{a}>0, \quad -\left(\frac{e}{b}\right)_L<\frac{\overline{d}}{\overline{a}}<\left(\frac{f}{c}\right)_L,
\end{equation}
and
\begin{equation}\label{19}
	\left(\frac{b}{e}\right)_L>\left(\frac{c}{f}\right)_M.
\end{equation}
The authors in \cite{AmiOrt1994} proved that if \eqref{19} holds, then \eqref{1} has at most one coexistence state \cite[Proposition $3.3$]{AmiOrt1994}. 

Moreover, as observed by the authors in \cite[Remark $2$, p. $11$]{AmiOrt1994}, if both \eqref{18} and \eqref{19} holds, then it follows from Tineo \cite[Theorem~$1.5$]{Tineo1} that \eqref{1} has exactly one coexistence state and it is \emph{globally} asymptotically stable. 

More sufficient conditions for asymptotically stability of coexistence stated were obtained by R. Ortega and coauthors \cites{AmiOrt1994,Ort2021}. Such conditions can somewhat be called \emph{$L_1$-condition} and \emph{$L_\infty$-condition} (see Remarks~\ref{Remark2} and \ref{Remark3}). In recent years V. Ortega and Rebelo~\cite{OrtReb2023} constructed a bridge between these two conditions, obtaining a \emph{$L_p$-condition}, $p\in[1,\infty]$. Such condition requires that the $L_p$-norm of all possible coexistence states $(u_0,v_0)$ to satisfy a given inequality. To this end, the authors in \cite{OrtReb2023} also provide upper bounds for $||u_0||_p$ and $||v_0||_p$, independent from each other. This independence provides a \emph{unified} sufficient condition for uniqueness and asymptotically stability of $(u_0,v_0)$. 

In this paper we obtain new \emph{intertwined} upper bounds for $||u_0||_p$ and $||v_0||_p$, which in turn implies on new sufficient conditions for uniqueness and asymptotically stability of $(u_0,v_0)$, that can be applied when the unified test is inconclusive.

The paper is organized as follows. In Section~\ref{Sec2} we state our main Theorem. At Section~\ref{Sec3} we have some preliminaries results to prove the main Theorem at Section~\ref{Sec4}. In Section~\ref{Sec5} we provide an example where previous results in the literature are inconclusive, while ours is not. We also provide some further thoughts. Finally, we have an Appendix with some technicalities and illustrations.

\section{Statement of the main result}\label{Sec2}

Given $p\in[1,\infty)$ and $\varphi\in L^p$ non-negative, the $L^p$-\emph{average} $\overline{\varphi}_p\in\mathbb{R}_{\geqslant0}$ is given by,
	\[\overline{\varphi}_p=\frac{1}{T^\frac{1}{p}}\left(\int_{0}^{T}\varphi^p\right)^\frac{1}{p}=\frac{1}{T^\frac{1}{p}}||\varphi||_p.\]
If $p=\infty$, then we define $\overline{\varphi}_\infty=||\varphi||_\infty$. Observe that $\overline{\varphi}_p\to\overline{\varphi}_\infty$ as $p\to\infty$.

Given a system of the form \eqref{1} and $p\in[1,\infty)$, let $C_p\subset\mathbb{R}^2$ be the set given by the points $(x,y)\in\mathbb{R}^2$ such that $x>0$, $y>0$ and
\begin{equation}\label{21}
	\begin{array}{rcl}
		b_LU^{1-p}x^p+c_LV^{1-p}y^p\leqslant &\overline{a}& \leqslant b_Mx+c_My \vspace{0.2cm} \\
		-e_Mx+f_LV^{1-p}y^p\leqslant &\overline{d}& \leqslant-e_LU^{1-p}x^p+f_My,
	\end{array}
\end{equation}
where, 
	\[U=\left(\frac{a}{b}\right)_M, \quad V=\left(\frac{d}{f}\right)_M+\left(\frac{e}{f}\right)_MU.\]	
Let also $C_\infty\subset\mathbb{R}^2$ be the set given by
	\[0<x\leqslant U, \quad 0<y\leqslant V.\]
For an illustration of $C_p$ and more details about the definition of $C_\infty$, see Appendix~\ref{AppB}. We observe that $C_p$ is bounded and that it may be empty. Let also $\mathcal{J}\colon[1,\infty)\to\mathbb{R}$ be given by,
	\[\mathcal{J}(q)=\int_{0}^{2\pi}\frac{1}{\bigl(|\cos\theta|^{2q}+|\sin\theta|^{2q}\bigr)^\frac{1}{q}}\;d\theta.\]
It follows from \cite[Proposition~$2.2$]{OrtReb2023} that,
	\[\lim\limits_{q\to\infty}\mathcal{J}(q)=8.\]
Hence, we can continuously extend $\mathcal{J}$ to $[1,\infty]$ by defining $\mathcal{J}(\infty)=8$. For more details about $\mathcal{J}(q)$, see Appendix~\ref{AppC}. Given $p\in[1,\infty]$, we recall that its \emph{conjugate} $q\in[1,\infty]$ is given by the unique solution of $1/p+1/q=1$. In what follows $q$ always denote the conjugate of $p$. Our main result is the following.

\begin{theorem}\label{Main1}
	Consider a system of the form \eqref{1} and its respective set $C_p$, $p\in[1,\infty]$. Then the following statements hold.
	\begin{enumerate}[label=(\alph*)]
		\item If $(u(t),v(t))$ is coexistence state, then $(\overline{u}_p,\overline{v}_p)\in C_p$. In particular, $C_p\neq\emptyset$ for every $p\in[1,\infty]$.
		\item Suppose we have at least one coexistence state. If
			\[T\left(\sqrt{c_Me_Mx_py_p}+\frac{1}{2}(b_Mx_1+f_My_1)\right)\leqslant\frac{\mathcal{J}(q)}{2^{2-\frac{1}{q}}}\]
		for every $(x_p,y_p)\in C_p$ and $(x_1,y_1)\in C_1$, then such coexistence state is unique and asymptotically stable.
	\end{enumerate}
\end{theorem}

\begin{remark}
	As presented in the introduction, we recall that the existence of a coexistence state is completely characterized. Therefore, the hypothesis of having at least one is not a loss of generality. 
\end{remark}

It follows from Hölder's inequality that if $\varphi\geqslant0$, then $\overline{\varphi}_{p_1}\leqslant\overline{\varphi}_{p_2}$ for $1\leqslant p_1\leqslant p_2\leqslant\infty$. Hence, we can replace statement $(b)$ of Theorem~\ref{Main1} for the following weak version.

\begin{enumerate}
	\item[$(b')$] Suppose we have at least one coexistence state. If
	\[T\left(\sqrt{c_Me_Mx_py_p}+\frac{1}{2}(b_Mx_p+f_My_p)\right)\leqslant\frac{\mathcal{J}(q)}{2^{2-\frac{1}{q}}}\]
	for every $(x_p,y_p)\in C_p$, then such coexistence state is unique and asymptotically stable.
\end{enumerate}

\begin{remark}[The $L_1$-condition]\label{Remark2}
	Observe that if we replace $p=1$ at Theorem~\ref{Main1}$(b)$, we obtain
		\[T\left(\sqrt{c_Me_Mxy}+\frac{1}{2}(b_Mx+f_My)\right)\leqslant2\]
	for every $(x,y)\in C_1$, where $C_1\subset\mathbb{R}^2$ is the set bounded by $x>0$, $y>0$ and
	\[\begin{array}{rcl}
		b_Lx+c_Ly\leqslant &\overline{a}& \leqslant b_Mx+c_My \vspace{0.2cm} \\
		-e_Mx+f_Ly\leqslant &\overline{d}& \leqslant-e_Lx^p+f_My.
	\end{array}\]
This is precisely the $L_1$-condition obtained by R. Ortega~\cite[Theorem~$5.2$]{Ort2021}.
\end{remark}

\begin{remark}[The $L_\infty$-condition]\label{Remark3}
	Observe that if we replace $p=\infty$, $x=U$ and $y=V$ at Theorem~\ref{Main1}$(b')$, we obtain
		\[T\left(\sqrt{c_Me_MUV}+\frac{1}{2}(b_MU+f_MV)\right)\leqslant\pi.\]
	This is the $L_\infty$-condition obtained by R. Ortega and Amine~\cite[Proposition~$4.5$]{AmiOrt1994}.
\end{remark}

\section{Preliminaries results}\label{Sec3}

In this section we recall the \emph{$L_p$-condition} obtained by V. Ortega and Rebelo~\cite[Theorem $3.1$]{OrtReb2023} and also a technical lemma.

\begin{theorem}[The $L_p$-condition]\label{T2}
	Suppose that all possible coexistence states $(u,v)$ of system \eqref{1} satisfy\footnote{Actually in their paper instead of the fraction $1/2$ in the expression, it appears the fraction $T/2$. But from equation $(7)$ in that paper one can see that it is a typo.}
	\begin{equation}\label{20}
		T^\frac{1}{q}\sqrt{||eu||_p||cv||_p}+\frac{1}{2}||bu-fv||_1\leqslant\frac{\mathcal{J}(q)}{2^{2-\frac{1}{q}}},
	\end{equation}
	where $p$ and $q$ are conjugated indices and $p$, $q\in[1,\infty]$. Then the coexistence state is unique and asymptotically stable. Moreover, any coexistence state $(u,v)$ satisfies
	\begin{equation}\label{26}
		||u||_p\leqslant\frac{||a||_p}{b_L}, \quad ||v||_p\leqslant\frac{||d||_p}{f_L}+\frac{e_M}{f_L}\frac{||a||_p}{b_L}.
	\end{equation}
\end{theorem}

\begin{lemma}\label{Lemma1}
	Let $(u(t),v(t))$ be a coexistence state of \eqref{1}. Then $||u||_\infty\leqslant U$ and $||v||_\infty\leqslant V$.
\end{lemma}

\begin{proof} Let $\tau\in[0,T]$ be such that $u(\tau)=\max_{[0,T]}u(t)$. Since $u(\tau)$ is a local maximum, it follows that $\dot u(\tau)=0$ and thus it follows from the first equation of \eqref{1} that
	\[a(\tau)=b(\tau)u(\tau)+c(\tau)v(\tau)\geqslant b(\tau)u(\tau) \Rightarrow u(\tau)\leqslant \frac{a(\tau)}{b(\tau)}\leqslant U.\]
Similarly, if we let $\tau\in[0,T]$ be such that $v(\tau)=\max_{[0,T]}v(t)$, then it follows from the second equation of \eqref{1} that,
	\[f(\tau)v(\tau)=d(\tau)+e(\tau)u(\tau) \Rightarrow v(\tau)\leqslant\frac{d(\tau)}{f(\tau)}+\frac{e(\tau)}{f(\tau)}u(\tau)\leqslant V.\]
This finishes the proof. \end{proof}

For lower bound of the coexistence states of \eqref{1}, we refer to \cite[Lemma $5.5$]{GomOrtTin1996}.

\section{Proof of Theorem~\ref{Main1}}\label{Sec4}

Before we prove the theorem, we observe that given $p\in[1,\infty]$ and $\varphi\in L^P$, it follows from Hölder's inequality that,
\begin{equation}\label{2}
	||\varphi||_1\leqslant T^{1-\frac{1}{p}}||\varphi||_p.
\end{equation}
Moreover, it follows from Littlewood's inequality that
	\[||\varphi||_p\leqslant||\varphi||_1^\frac{1}{p}||\varphi||_\infty^{1-\frac{1}{p}},\]
and thus
\begin{equation}\label{3}
	||\varphi||_1\geqslant\frac{1}{||\varphi||_\infty^{p-1}}||\varphi||_p^p.
\end{equation}

\begin{proof}[Proof of Theorem~\ref{Main1}] Let $(u(t),v(t))$ be a coexistence state of \eqref{1}. Dividing the first equation of \eqref{1} by $u$ we obtain,
\begin{equation}\label{5}
	\frac{\dot u}{u}=a(t)-b(t)u-c(t)v.
\end{equation}
Integrating \eqref{5} in $t$, from $0$ to $T$, we obtain
	\[\int_{0}^{T}a=\int_{0}^{T}bu+\int_{0}^{T}cv,\]
and thus,
\begin{equation}\label{6}
	 b_L\int_{0}^{T}u+c_L\int_{0}^{T}v\leqslant \int_{0}^{T}a \leqslant b_M\int_{0}^{T}u+c_M\int_{0}^{T}v.
\end{equation}
Since $u\geqslant0$ and $v\geqslant0$, it follows from \eqref{6} that,
\begin{equation}\label{7}
	b_L||u||_1+c_L||v||_1\leqslant\int_{0}^{T}a\leqslant b_M||u||_1+c_M||v||_1.
\end{equation}
Applying \eqref{2} on the right-hand side of \eqref{7} we obtain
\begin{equation}\label{28}
	\int_{0}^{T}a\leqslant b_MT^{1-\frac{1}{p}}||u||_p+c_MT^{1-\frac{1}{p}}||v||_p.
\end{equation}
Dividing \eqref{28} by $T$ and knowing that $T^{-\frac{1}{p}}||u||_p=\overline{u}_p$ and $T^{-\frac{1}{p}}||v||_p=\overline{v}_p$ we obtain,
\begin{equation}\label{8}
	\overline{a}\leqslant b_M\overline{u}_p+c_M\overline{v}_p.
\end{equation}
Similarly, applying \eqref{3} on the left-hand side of \eqref{7} we obtain,
	\[\int_{0}^{T}a\geqslant b_L\frac{1}{||u||_\infty^{p-1}}||u||_p^p+c_L\frac{1}{||v||_\infty^{p-1}}||v||_p^p.\]
Hence, it follows from Lemma~\ref{Lemma1} that,
\begin{equation}\label{9}
	\int_{0}^{T}a\geqslant b_L\frac{1}{U^{p-1}}||u||_p^p+c_L\frac{1}{V^{p-1}}||v||_p^p.
\end{equation}
Multiplying \eqref{9} by $T^{-1}$ we obtain,
\begin{equation}\label{10}
	\overline{a}\geqslant b_L\frac{1}{U^{p-1}}\overline{u}_p^p+c_L\frac{1}{V^{p-1}}\overline{v}_p^p.
\end{equation}
Let us now look to the second equation of \eqref{1}. Dividing it by $v$ and integrating it in $t$, from $0$ to $T$, we obtain
	\[\int_{0}^{T}d=-\int_{0}^{T}eu+\int_{0}^{T}fv,\]
and thus,
\begin{equation}\label{11}
	-e_M||u||_1+f_L||v||_1\leqslant\int_{0}^{T}d\leqslant-e_L||u||_1+f_M||v||_1.
\end{equation}
Applying \eqref{2} (resp. \eqref{3} and Lemma~\ref{Lemma1}) on the positive (resp. negative) term of the right-hand side of \eqref{11} we obtain,
\begin{equation}\label{12}
	\int_{0}^{T}d\leqslant-e_L\frac{1}{U^{p-1}}||u||_p^p+f_MT^{1-\frac{1}{p}}||v||_p.
\end{equation}
Multiplying \eqref{12} by $T^{-1}$ we obtain,
\begin{equation}\label{13}
	\overline{d}\leqslant-e_L\frac{1}{U^{p-1}}\overline{u}_p^p+f_M\overline{v}_p.
\end{equation}
Similarly, it follows from the left-hand side of \eqref{11} that,
\begin{equation}\label{14}
	\overline{d}\geqslant-e_M\overline{u}_p+f_L\frac{1}{V^{p-1}}\overline{v}_p^p.
\end{equation}
Now observe that equations \eqref{8}, \eqref{10}, \eqref{13} and \eqref{14} are the four equations given at the definition \eqref{21} of the set $C_p$. Hence, we obtained statement $(a)$ of Theorem~\ref{Main1}.

Statement $(b)$ follows directly from Theorem~\ref{T2}. More precisely, knowing that 
	\[||eu||_p\leqslant e_M||u||_p, \quad ||cv||_p\leqslant c_M||v||_p, \quad ||bu-fv||_1\leqslant b_M||u||_1+f_M||v||_p,\]
and that $T^\frac{1}{q}=T/T^\frac{1}{p}$, one can see that the left-hand side of \eqref{20} can be majored by,
	\[T\left(\sqrt{c_Me_M\overline{u}_p\overline{v}_p}+\frac{1}{2}(b_M\overline{u}_1+f_M\overline{v}_1)\right).\]
Since every possible coexistence state $(u,v)$ satisfies $(\overline{u}_p,\overline{v}_p)\in C_p$ it follows from Theorem~\ref{T2} that if 
	\[T\left(\sqrt{c_Me_Mx_py_p}+\frac{1}{2}(b_Mx_1+f_My_1)\right)\leqslant\frac{\mathcal{J}(q)}{2^{2-\frac{1}{q}}}\]
for every $(x_p,y_p)\in C_p$ and $(x_1,y_1)\in C_1$, then the coexistence state is unique and asymptotically stable. \end{proof}

\section{An example and further thoughts}\label{Sec5}

At first glance, one could look at the proof of Theorem~\ref{Main1} and conclude that since the boundaries of $C_p$ are obtained by majoration and minoration of the boundaries of $C_1$ (which in turn was obtained by R. Ortega~\cite[p. $11$]{Ort2021}), then no new information could be drawn from such theorem. However, as the next example will show, this is not the case. 

\begin{example}\label{Example1}
	Consider a Lotka-Volterra system
	\begin{equation}\label{4}
		\dot u=u(a-bu-cv), \quad \dot v=v(d+eu-fv),
	\end{equation}
	with the coefficients being positive constants. We recall the sufficient condition 
	\begin{equation}\label{22}
		T\left(\sqrt{c_Me_Mx_py_p}+\frac{1}{2}(b_Mx_1+f_My_1)\right)\leqslant\frac{\mathcal{J}(q)}{2^{2-\frac{1}{q}}}
	\end{equation}
	of Theorem~\ref{Main1}$(b)$. Let $F(q)=\mathcal{J}(q)/2^{2-\frac{1}{q}}$, where $q=p/(p-1)$ is the conjugate of $p$, and let also $\mathcal{F}(p)=F(p/(p-1))$ (see Appendix~\ref{AppC}). Since the coefficients of \eqref{4} are constants, it is not hard to see that $C_1=\{(x_1,y_1)\}$ is a single point, given by the unique solution of
		\[\left(\begin{array}{cc} b & c \\ -e & f \end{array}\right)\left(\begin{array}{c} x_1 \\ y_1 \end{array}\right)=\left(\begin{array}{c} a \\ d \end{array}\right).\]
	Let $k=(b_Mx_1+f_My_1)/2$ and observe that we can rewrite \eqref{22} as $x_p^py_p^p\leqslant h(p)$, where
		\[h(p)=\left[\frac{1}{\sqrt{ce}}\left(\frac{\mathcal{F}(p)}{T}-k\right)\right]^{2p}.\]
	We recall that $C_p\subset\mathbb{R}^2$ is the set bounded by $x>0$, $y>0$ and
	\[\begin{array}{rcl}
		b_LU^{1-p}x^p+c_LV^{1-p}y^p\leqslant &\overline{a}& \leqslant b_Mx+c_My \vspace{0.2cm} \\
		-e_Mx+f_LV^{1-p}y^p\leqslant &\overline{d}& \leqslant-e_LU^{1-p}x^p+f_My,
	\end{array}\]
	if $p\in[1,\infty)$, and by
		\[0<x\leqslant U, \quad 0<y\leqslant V,\]
	if $p=\infty$ (for an illustration, see Appendix~\ref{AppB}). Therefore, since the left-hand side of \eqref{22} is an increasing function of $x_p$ and $y_p$ (recall that $x_1$ and $y_1$ are constants), it follows that its maximum happens somewhere on the curve
		\[bU^{1-p}x_p^p+cV^{1-p}y_p^p=a.\]
	Hence, the maximum occurs at some point $(x_p,y_p)$ such that
	\begin{equation}\label{23}
		y_p^p=c^{-1}V^{p-1}(a-bU^{1-p}x_p^p).
	\end{equation}	
	Replacing \eqref{23} at $x_p^py_p^p\leqslant h(p)$ we obtain,
	\begin{equation}\label{24}
		-\left[\frac{b}{c}\left(\frac{V}{U}\right)^{p-1}\right]w^2+\left[\frac{a}{c}V^{p-1}\right]w-h(p)\leqslant0,
	\end{equation}
	where $w=x_p^p$. The discriminant of \eqref{24} in relation to $w$ is given by,
		\[\Delta(p)=V^{p-1}\left[\frac{a^2}{c^2}V^{p-1}-4\frac{b}{c}\frac{1}{U^{p-1}}h(p)\right].\]
	Therefore, to study its sign it is sufficient to study the sign of	
		\[G(p)=\frac{a^2}{c^2}V^{p-1}-4\frac{b}{c}\frac{1}{U^{p-1}}h(p).\]
	It follows from this reasoning that \eqref{22} is equivalent to \eqref{24}. Since the leading coefficient of the left-hand side of \eqref{24} is negative, it follows that the inequality holds at a given $p\in[1,\infty)$ if, and only if, $G(p)\leqslant0$. 
	
	Hence we conclude that we can apply Theorem~\ref{Main1} at \eqref{4} at a given $p\in[1,\infty)$ if, and only if, $G(p)\leqslant0$. Therefore, if for some choice of the coefficients it holds that
	\begin{equation}\label{25}
		G(1)>0, \quad G(p^*)<0, \quad \lim\limits_{p\to\infty}G(p)>0,
	\end{equation}
	for some $p^*\in(1,\infty)$, then the $L_1$ and $L_\infty$-conditions (recall Remarks~\ref{Remark2} and \ref{Remark3}) will be inconclusive. Moreover, since the coefficients are constant, it follows that the upper bounds \eqref{26} of the $L_p$-condition is reduced to $\overline{u}_p\leqslant U$, $\overline{v}_p\leqslant V$ and thus coincide with the $L_\infty$-condition. Therefore, we conclude that if \eqref{25} holds, then the previous results in the literature are inconclusive, while Theorem~\ref{Main1} can still be applied at $p=p^*$ and thus ensure the uniqueness and asymptotically stability of the coexistence state.
	
	For a example of such situation, one can consider $T=1$ and the values 
	\begin{equation}\label{30}
		a=2.0102, \quad b=1, \quad c=0.0051, \quad d=2.0203, \quad e=0.9898, \quad f=2,
	\end{equation}
	and see that \eqref{25} holds with $p^*=2$.
\end{example}

\begin{remark}
	We have assumed constant coefficients at Example~\ref{Example1} for the sake of simplicity. However, it follows from the continuous dependence of the boundaries of $C_p$, $p\in[1,\infty]$, in relation to the coefficient functions $a(t),\dots,f(t)$ that if the amplitude of such coefficients (i.e. $a_M-a_L,\dots,f_M-a_L$) are small enough and if their averages are close enough to \eqref{30}, then the same conclusion holds. Moreover, the assumption $T=1$ is no loss of generality since it can be obtained by time rescaling.
\end{remark}

We think that one reason for Theorem~\ref{Main1} to be conclusive in some situations while the previous one are not is the fact that although the boundaries of $C_p$ are obtained by inequalities on the boundaries of $C_1$, the function $\mathcal{F}(p)$ \emph{also increases} (see Appendix~\ref{AppC}). Therefore, for a given $p\in(1,\infty)$, the function $\mathcal{F}(p)$ may have increased in such a way that it compensates the boundaries of $C_p$.

Moreover in comparison with the boundaries provided by \eqref{26}, as anticipated in the introduction, the fact that the two inequalities are independent from each other provides a \emph{unified} test for uniqueness and asymptotically stability of the coexistence states. More precisely, if
	\[T^\frac{1}{q}\sqrt{c_Me_M\alpha_p\beta_p}+\frac{1}{2}(b_M\alpha_1+f_M\beta_1)\leqslant\frac{\mathcal{J}(q)}{2^{2-\frac{1}{q}}}\]
for some $p\in[1,\infty]$, where
	\[\alpha_p=\frac{||a||_p}{b_L}, \quad \beta_p=\frac{||d||_p}{f_L}+\frac{e_M}{f_L}\frac{||a||_p}{b_L},\]
then we are done. However, if the unified test is inconclusive then the intertwined boundaries \eqref{21} of $C_p$ can reduce the lack of possibilities of the averages of the coexistence states in such a way that for all \emph{possible} coexistence states, inequality \eqref{20} still holds.

\section{Acknowledgments}

The author is grateful for the hospitality of Universidad de Granada, where he had many stimulating conversations with professor R. Ortega.

This work is supported by S\~ao Paulo Research Foundation (FAPESP), grants 2019/10269-3, 2021/01799-9 and 2022/14353-1.

\appendix

\section{Linearly stable, asymptotic stable and Floquet Theory}\label{AppA}

Consider a system $T$-periodic system of differential equations
\begin{equation}\label{15}
	\dot u=H(t,u),
\end{equation}
with $u\in\mathbb{R}^n$ and $H\colon\mathbb{R}\times\mathbb{R}^n\to\mathbb{R}^n$ sufficiently smooth. Let $u_0(t)$ be a $T$-periodic solution of \eqref{15}. The \emph{linearization} of \eqref{15} at $u_0(t)$ is the $T$-periodic linear system given by
\begin{equation}\label{16}
	\dot u=D_uH(t,u_0(t))u,
\end{equation}
where $D_uH$ denotes the Jacobian matrix of $H$ in relation to the variable $u$. We say that $u_0(t)$ is \emph{linearly stable} if all \emph{Floquet exponents} (also known as \emph{characteristics exponents}) associated to \eqref{16} have non-positive real part. Fore more details about Floquet Theory, we refer to \cite[Section $3.5$]{CL} and \cite{NovaesPereira}.

We say that $u_0$ is \emph{stable} if for every neighborhood $V\subset\mathbb{R}^n$ of $u_0(0)$ there is a neighborhood $W\subset V$ of $u_0(0)$ such that if $u(t)$ is a solution of \eqref{15} with $u(0)\in W$, then $u(t)\in V$ for $t\geqslant0$. Moreover, given an open set $U\subset\mathbb{R}^n$, we say that $u_0(t)$ is \emph{asymptotically stable} in $U$ if $u_0$ is stable and
	\[\lim\limits_{t\to\infty}|u_0(t)-u(t)|=0,\]
for every solution $u(t)$ of \eqref{15} with $u(0)\in U$.

\section{Illustrations of the set $C_p$ and construction of $C_\infty$}\label{AppB}

We recall that $C_p\subset\mathbb{R}^2$ is the set bounded by $x>0$, $y>0$ and
\[\begin{array}{rcl}
	b_LU^{1-p}x^p+c_LV^{1-p}y^p\leqslant &\overline{a}& \leqslant b_Mx+c_My \vspace{0.2cm} \\
	-e_Mx+f_LV^{1-p}y^p\leqslant &\overline{d}& \leqslant-e_LU^{1-p}x^p+f_My,
\end{array}\]
for $p\in[1,\infty)$, and by
	\[0<x\leqslant U, \quad 0<y\leqslant V,\]
for $p=\infty$. See Figure~\ref{Fig1}. 
\begin{figure}[h]
	\begin{center}
		\begin{minipage}{6cm}
			\begin{center}
				\begin{overpic}[height=4cm]{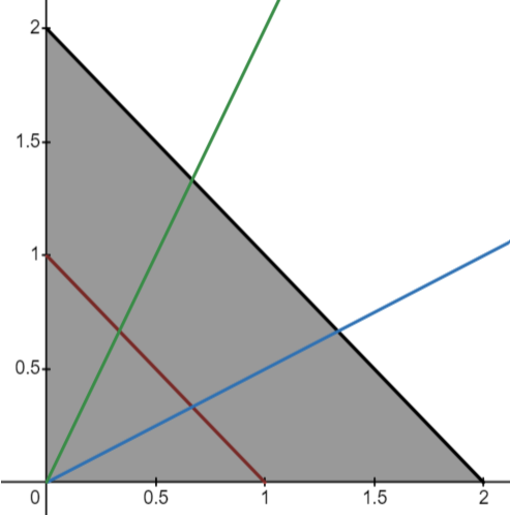} 
				\end{overpic}
				
				$p=1$.
			\end{center}
		\end{minipage}
		\begin{minipage}{6cm}
			\begin{center}
				\begin{overpic}[height=4cm]{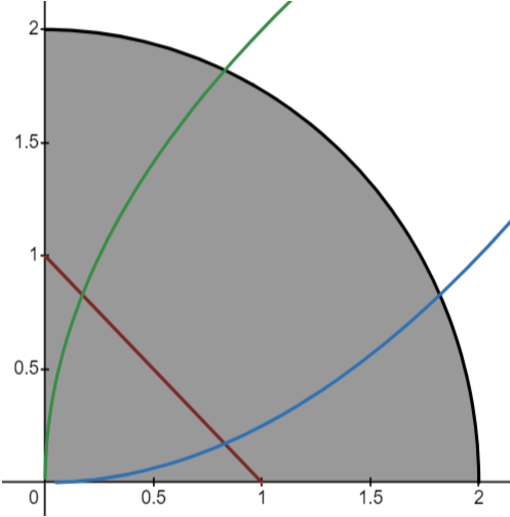} 
				\end{overpic}
				
				$p=2$.
			\end{center}
		\end{minipage}
	\end{center}
$\;$
	\begin{center}
		\begin{minipage}{6cm}
			\begin{center}
				\begin{overpic}[height=4cm]{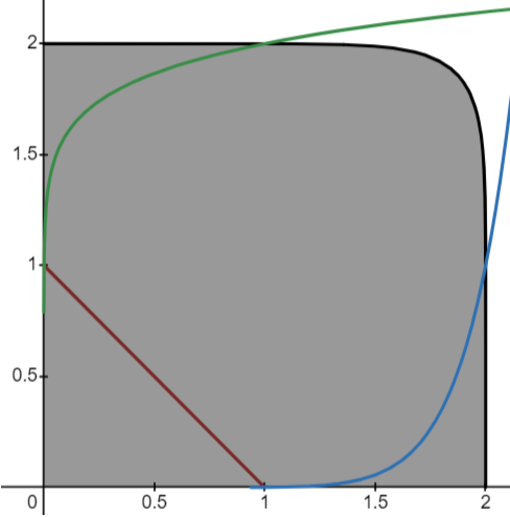} 
				\end{overpic}
				
				$p=10$.
			\end{center}
		\end{minipage}
		\begin{minipage}{6cm}
			\begin{center}
				\begin{overpic}[height=4cm]{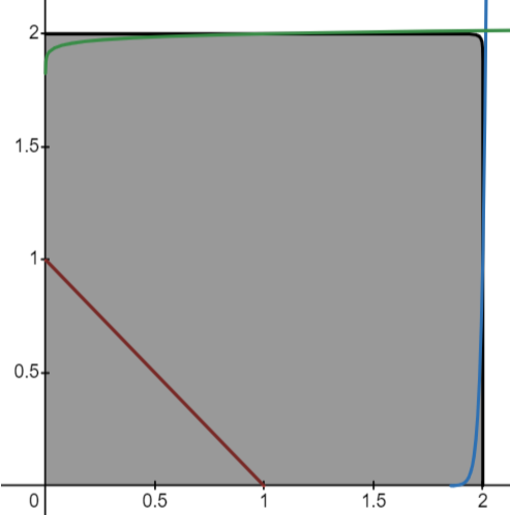} 
				\end{overpic}
				
				$p=100$.
			\end{center}
		\end{minipage}
	\end{center}
	\caption{Illustrations of $C_p$, given by the region bounded by the four curves, with parameters $\overline{a}=2$, $\overline{d}=2$, $b_M=2$, $b_L=1$, $c_M=2$, $c_L=1$, $e_M=2$, $e_L=1$, $f_M=2$, $f_L=1$, $U=2$ and $V=2$. The black region (resp. curve) corresponds to the inequality (resp. equality) $b_LU^{1-p}x^p+c_LV^{1-p}y^p\leqslant\overline{a}$. The red curve correspond to $b_Mx+c_My=\overline{a}$, while the green the blue ones correspond to $-e_Mx+f_LV^{1-p}y^p=\overline{d}$ and $-e_LU^{1-p}x^p+f_My=\overline{d}$, respectively. Together, they form the boundaries of $C_p$. Colors available in the online version.}\label{Fig1}
\end{figure}
To obtain the definition of $C^\infty$, consider the inequality
\begin{equation}\label{27}
	b_LU^{1-p}x^p+c_LV^{1-p}y^p\leqslant\overline{a}.
\end{equation}
Since the left-hand side is strictly increasing in $x$ and $y$ it is clear that the maximum value of $x$ occurs when $y=0$. Hence, replacing $y=0$ at \eqref{27} we obtain,
	\[x\leqslant\left(\frac{\overline{a}}{b_L}\right)^\frac{1}{p}U^\frac{1}{q}.\] 
Taking the limit $p\to\infty$ we obtain $x\leqslant U$.
Similarly, by replacing $x=0$ at \eqref{27} and taking the limit $p\to\infty$ one obtains $y\leqslant V$.

\section{Properties of the maps $\mathcal{J}$, $F$ and $\mathcal{F}$.}\label{AppC}

We recall that $\mathcal{J}\colon[1,\infty)\to\mathbb{R}$ is the function given by,
\begin{equation}\label{29}
	\mathcal{J}(q)=\int_{0}^{2\pi}\frac{1}{\bigl(|\cos\theta|^{2q}+|\sin\theta|^{2q}\bigr)^\frac{1}{q}}\;d\theta.
\end{equation}
Moreover, it follows from \cite[Proposition~$2.2$]{OrtReb2023} that
	\[\lim\limits_{q\to\infty}\mathcal{J}(q)=8,\]
and thus we can continuously extend $\mathcal{J}$ to $[1,\infty]$ by defining $\mathcal{J}(\infty)=8$. We also recall that $F$, $\mathcal{F}\colon\mathbb[1,\infty]\to\mathbb{R}$ are the functions given by 
\begin{equation}\label{33}
	F(q)=\frac{\mathcal{J}(q)}{2^{2-\frac{1}{q}}}, \quad \mathcal{F}(p)=F\left(\frac{p}{p-1}\right)=\dfrac{\mathcal{J}\left(\frac{p}{p-1}\right)}{2^{1+\frac{1}{p}}},
\end{equation}
where $q=p/(p-1)$ is the conjugate of $p$. For a numerical plot of the graph of $\mathcal{F}(p)$, see Figure~\ref{Fig2}.
\begin{figure}[h]
	\begin{center}
		\begin{minipage}{12cm}
			\begin{center}
				\begin{overpic}[height=4cm]{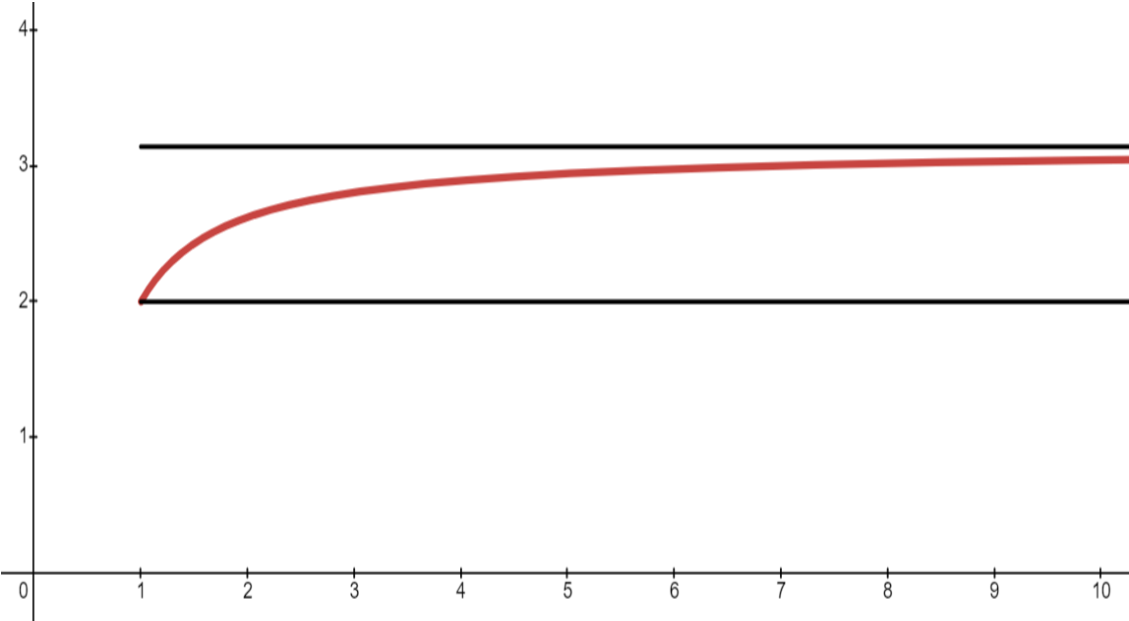} 
				\end{overpic}
			\end{center}
		\end{minipage}
	\end{center}
	\caption{Graph of the function $\mathcal{F}(p)$, in red. The black straight lines are the lower and upper bounds $y=2$ and $y=\pi$, respectively. Colors available in the online version.}\label{Fig2}
\end{figure}

\begin{proposition}\label{Prop1}
	The functions $F$ and $\mathcal{F}$ are strictly decreasing and increasing, respectively.
\end{proposition}

\begin{proof} First we consider the function $F$. That is, we consider
\begin{equation}\label{35}
	F(q)=\frac{\mathcal{J}(q)}{2^{2-\frac{1}{q}}}=\int_{0}^{2\pi}2^{\frac{1}{q}-2}\bigl(|\cos\theta|^{2q}+|\sin\theta|^{2q}\bigr)^{-\frac{1}{q}}\;d\theta.
\end{equation}
Consider the integrand of the outermost right-hand side of \eqref{35},
	\[g_1(\theta,q)=2^{\frac{1}{q}-2}\bigl(|\cos\theta|^{2q}+|\sin\theta|^{2q}\bigr)^{-\frac{1}{q}}.\]
Differentiating it in relation to $q$ we obtain	
\begin{equation}\label{34}
	\frac{\partial g_1}{\partial q}(\theta,q)=g_1(\theta,q)\left[\frac{1}{q^2(\cos^{2q}\theta+\sin^{2q}\theta)}g_2(\theta,q)+\ln\bigl(2^{-1/q^2}\bigr)\right],
\end{equation}
where
\begin{equation}\label{36}
	\begin{array}{l}
		\displaystyle g_2(\theta,q)= \cos^{2q}\theta\bigl[\ln(\cos^{2q}\theta+\sin^{2q}\theta)-\ln\cos^{2q}\theta\bigr] \vspace{0.2cm} \\
		\displaystyle \qquad\qquad\qquad + \sin^{2q}\theta\bigl[\ln(\cos^{2q}\theta+\sin^{2q}\theta)-\ln\sin^{2q}\theta\bigr].
	\end{array}
\end{equation}
Since $g_1$ is strictly positive, it follows that the sign of \eqref{34} is given by the sign of,
\begin{equation}\label{38}
	\frac{1}{q^2(\cos^{2q}\theta+\sin^{2q}\theta)}g_2(\theta,q)+\ln\bigl(2^{-1/q^2}\bigr).
\end{equation}
We claim that \eqref{38} is negative. Indeed, observe that \eqref{38} is negative if and only if,
	\[g_2(\theta,q)\leqslant -q^2(\cos^{2q}\theta+\sin^{2q}\theta)\ln\bigl(2^{-1/q^2}\bigr).\]
Which by properties of the logarithm is equivalent to,
	\[g_2(\theta,q)\leqslant (\cos^{2q}\theta+\sin^{2q}\theta)\ln2.\]
Which in turn is equivalent to,
\begin{equation}\label{31}
	\begin{array}{l}
		\displaystyle g_3(\theta,q):= \cos^{2q}\bigl[\ln(\cos^{2q}\theta+\sin^{2q}\theta)-\ln(2\cos^{2q}\theta)\bigr] \vspace{0.2cm} \\
		\displaystyle  \qquad\qquad\qquad +
		\sin^{2q}\bigl[\ln(\cos^{2q}\theta+\sin^{2q}\theta)-\ln(2\sin^{2q}\theta)\bigr]\leqslant0.
	\end{array}
\end{equation}
Differentiating \eqref{31} in relation to $\theta$ we obtain,
\begin{equation}\label{32}
	\begin{array}{l}
		\displaystyle \frac{\partial g_3}{\partial\theta}(\theta,q)= 2q\sin\theta\cos\theta\Bigl[\cos^{2(q-1)}\theta\bigl[\ln(2\cos^{2q}\theta)-\ln(\cos^{2q}\theta+\sin^{2q}\theta)\bigr] \vspace{0.2cm} \\
		\displaystyle  \qquad\qquad\qquad\qquad\qquad\qquad -\sin^{2(q-1)}\theta\bigl[\ln(2\sin^{2q}\theta)-\ln(\cos^{2q}\theta+\sin^{2q}\theta)\bigr]\Bigr].
	\end{array}
\end{equation}
Clearly, the zeros of \eqref{32} in $\theta$ are given by $\theta\in\{0,\pi/2,\pi,3\pi/2\}$ and by the zeros of
\begin{equation}\label{37}
	\begin{array}{l}
		\displaystyle g_4(\theta,q):= \cos^{2(q-1)}\theta\bigl[\ln(2\cos^{2q}\theta)-\ln(\cos^{2q}\theta+\sin^{2q}\theta)\bigr] \vspace{0.2cm} \\
		\displaystyle  \qquad\qquad\qquad\qquad -\sin^{2(q-1)}\theta\bigl[\ln(2\sin^{2q}\theta)-\ln(\cos^{2q}\theta+\sin^{2q}\theta)\bigr].
	\end{array}
\end{equation}
Observe that if $\theta_0\in[0,2\pi]$ is such that $\cos^2\theta_0=\sin^2\theta_0$, then $g_4(\theta_0,q)=0$. Reciprocally, we claim that if $\cos^2\theta_0\neq0\sin^2\theta_0$, then $g_4(\theta_0,q)\neq0$. Indeed, if $\cos^2\theta_0>\sin^2\theta_0$, then it follows from the fact that the logarithm is a strictly increasing function that,
	\[\ln(2\cos^{2q}\theta_0)>\ln(\cos^{2q}\theta_0+\sin^{2q}\theta_0), \quad \ln(2\sin^{2q}\theta_0)<\ln(\cos^{2q}\theta_0+\sin^{2q}\theta_0).\]
Hence, $g_4(\theta_0,q)>0$. Similarly if $\cos^2\theta_0<\sin^2\theta_0$, then $g_4(\theta_0,q)<0$.

Therefore, we conclude that the extreme points $\theta_0\in[0,2\pi]$ of \eqref{31} are such that $\cos\theta_0\sin\theta_0=0$ or $\cos^2\theta_0=\sin^2\theta_0$. In any case, it is easy to see that \eqref{31} holds (with the equality holding only if $\cos^2\theta_0=\sin^2\theta_0$), which in turn implies that the same holds for \eqref{34}. Therefore, it follows from Leibniz integral rule that
	\[F'(q)=\int_{0}^{2\pi}\frac{\partial g_1}{\partial q}(\theta,q)\;d\theta<0.\]
That is, $F'(q)$ is strictly decreasing. Since $\mathcal{F}(p)=F(p/(p-1))$, it follows that
	\[\mathcal{F}'(p)=-\frac{1}{(p-1)^2}F'\left(\frac{p}{p-1}\right)>0,\]
and thus $\mathcal{F}$ is strictly increasing. \end{proof}

\begin{remark}
	With the same reasoning of Proposition~\ref{Prop1}, one can prove that $\mathcal{J}(q)$ is also strictly increasing. One just need to consider the integrand
		\[h_1(\theta,q)=\frac{1}{\bigl(|\cos\theta|^{2q}+|\sin\theta|^{2q}\bigr)^\frac{1}{q}},\]
	observe that
		\[\frac{\partial h_1}{\partial q}(\theta,q)=\frac{h_1(\theta,q)}{q^2(\cos^{2q}\theta+\sin^{2q}\theta)}g_2(\theta,q),\]
	with $g_2$ given by \eqref{36}, and observe that since the logarithm is strictly increasing, it follows that
		\[\ln(\cos^{2q}\theta+\sin^{2q}\theta)\geqslant\ln\cos^{2q}\theta, \quad \ln(\cos^{2q}\theta+\sin^{2q}\theta)\geqslant\ln\sin^{2q}\theta,\]
	with equality holding only if $\cos\theta\sin\theta=0$. The proof now follows from Leibniz integral rule.
\end{remark}

\end{document}